\documentclass[reqno]{amsart}
\usepackage{hyperref}
\usepackage{fancyhdr}
\usepackage[dvips]{color}
\usepackage{amsfonts}
\usepackage{graphicx}
\usepackage{amscd}
\usepackage{mathrsfs}
\usepackage{amsmath}
\usepackage{amssymb}
\usepackage[latin1]{inputenc}
\theoremstyle{plain}

\newtheorem{corollary}{\bf Corollary}

\newtheorem{lemma}{\bf Lemma}

\newtheorem{proposition}{\bf Proposition}
\newtheorem{remark}{Remark}

\newtheorem{theorem}{\bf Theorem}
\newtheorem{conjecture}{Conjecture}{}
\numberwithin{equation}{section}

\newcommand\dm{\mathrm{dm}}
\newcommand\dn{\mathrm{d\mu}}
\newcommand{\g}[2]{\mbox{$\langle #1 ,#2 \rangle$}}
\newcommand\dv{\mathrm{div}}

\newcommand\dM{\mathrm{dM}}
\begin{document}

\title[Eigenvalue estimates for a class of elliptic differential operators]{Eigenvalue estimates for a class of elliptic differential operators
in divergence form}

\author{Jos\'e. N.V. Gomes$^1$}
\author{Juliana F.R. Miranda$^2$}
\address{$^{1,2}$Departamento de Matem\'atica, Universidade Federal do Amazonas, Av. General Rodrigo Oct\'avio, 6200, 69080-900 Manaus, Amazonas, Brazil.}
\address{$^1$Current: Department of Mathematics, Lehigh University, Christmas-Saucon Hall, 14 East Packer Avenue, Bethlehem, Pennsylvania 18015-3174, US.}

\email{$^1$jov217@lehigh.edu, jnvgomes@pq.cnpq.br}
\email{$^2$jfrmiranda@ufam.edu.br}

\urladdr{$^1$https://www1.lehigh.edu}
\urladdr{$^{1,2}$http://dmice.ufam.edu.br}


\begin{abstract}
We compute estimates for eigenvalues of a class of linear second-order elliptic differential operators in divergence form (with Dirichlet boundary condition) on a bounded domain in a complete Riemannian manifold. Our estimates are based upon the Weyl's asymptotic formula. As an application, we find a lower bound for the mean of the first k eigenvalues of the drifting Laplacian. In particular, we have extended for this operator a partial solution given by Cheng and Yang for the generalized conjecture of P\'olya. We also derive a second-Yang type inequality due to Chen and Cheng, and other two inequalities which generalize results by Cheng and Yang obtained for a domain in the unit sphere and for a domain in the projective space.
\end{abstract}
\maketitle

\section{Introduction}
In this paper $(M,\langle,\rangle)$ is a complete Riemannian manifold and the domain $\Omega\subset M$ is connected bounded with smooth boundary $\partial\Omega$ in $M$. Let $T$ be a symmetric positive definite $(1,1)$--tensor on $M$ and $\eta\in\mathcal{C}^2(M)$. We are interested in studying the eigenvalue problem with Dirichlet boundary condition, namely:
\begin{equation*}
\left\{
\begin{array}{ccccc}
-\mathscr{L}u &=& \lambda u &\hbox{in} & \Omega,\\
u&=& 0 & \hbox{on} & \partial\Omega,
\end{array}
\right.
\end{equation*}
where
\begin{equation}\label{defLintro}
\mathscr{L}u=\dv (T(\nabla u))-\langle\nabla\eta, T(\nabla u)\rangle.
\end{equation}

In this more general setting, we apply known techniques to prove our first result.

\begin{theorem}\label{propCYT-c3}
Let $\Omega$ be a domain in an $n$-dimensional complete Riemannian manifold $M$ isometrically immersed in $\mathbb{R}^m$, $\lambda_i$ be the $i$-th eigenvalue of  $\mathscr{L}$ and $u_i$ its corresponding normalized real-valued eigenfunction. Then
\begin{eqnarray*}
\mathrm{tr}(T)\sum_{i=1}^k(\lambda_{k+1}-\lambda_i)^2 &\leq& \sum_{i=1}^k(\lambda_{k+1}-\lambda_i)\Big((m-n)^2A_0^2T_*^2+(T_0+T_*\eta_0)^2\\
&&+4(T_0+T_*\eta_0)||T(\nabla u_i)||_{L^2(\Omega,\dm)}+4\lambda_i\Big),
\end{eqnarray*}
where $A_0=\max\{\sup_{\bar{\Omega}}|A_{e_k}|,\, k=n+1,\ldots,m\}$, $A_{e_k}$ is the Weingarten operator of the immersion with respect to $e_k$, $\eta_0=\sup_{\bar{\Omega}}|\nabla\eta|$, $T_*=\sup_{\bar{\Omega}}|T|$ and $T_0=\sup_{\bar{\Omega}}|\mathrm{tr}(\nabla T)|$.
\end{theorem}

The drifting Laplacian case $L= \Delta -\langle\nabla\eta, \nabla\cdot\rangle$ is recovered when $T$ is the identity operator in $\mathfrak{X}(M)$. In Section~\ref{sec-eta-laplaciano} we work specifically in this case. The main results for $L$ are described in the following two theorems.

\begin{theorem}\label{thmA3-c3}
Let $\Omega$ be a domain in an $n$-dimensional complete Riemannian manifold $M$ isometrically immersed in $\mathbb{R}^m$ with mean curvature $H$, and $\lambda_i$ be the $i$-th eigenvalue of the drifting Laplacian. Then
\begin{equation*}
\frac{1}{k}\sum_{i=1}^k\upsilon_i\geq \frac{n}{\sqrt{(n+2)(n+4)}}\frac{4\pi^2}{(\omega_n \mathrm{vol}\,\Omega)^\frac{2}{n}}k^\frac{2}{n}, \quad\mbox{for}\quad k=1,2,\ldots
\end{equation*}
where $\upsilon_i:=\lambda_i +\frac{n^2H_0^2+\eta_0^2+2\bar{\eta_0}}{4}$, $\eta_0=\sup_{\bar{\Omega}}|\nabla \eta|$, $\bar{\eta_0}=\sup_{\bar{\Omega}}|L\eta|$ and $H_0=\sup_{\bar{\Omega}}\|H\|.$
\end{theorem}

Theorem~\ref{thmA3-c3} is an extension for $L$ of Theorem~1.1 in Cheng and Yang~\cite{Cheng-Yang-III}. In particular, we have extended for the drifting Laplacian a partial solution given by them to the generalized conjecture of P\'olya (see Conjecture~\ref{conjecture}).

\begin{theorem}\label{thmA2-c3}
Under the assumptions in Theorem~\ref{thmA3-c3}, we have
\begin{equation}\label{cor.ineq1-c3}
\upsilon_{k+1}\leq\frac{1}{k}\Big(1+\frac{4}{n}\Big)\sum_{i=1}^k\upsilon_i,
\end{equation}
\begin{equation}\label{cor.ineq2-c3}
\upsilon_{k+1}\leq\Big(1+\frac{2}{n}\Big)\frac{1}{k}\sum_{i=1}^k\upsilon_i+\Big[\Big(\frac{2}{n}\frac{1}{k}\sum_{i=1}^k\upsilon_i\Big)^2
-\Big(1+\frac{4}{n}\Big)\frac{1}{k}\sum_{j=1}^k\Big(\upsilon_j-\frac{1}{k}\sum_{i=1}^k\upsilon_i\Big)^2\Big]^\frac{1}{2},
\end{equation}
\begin{equation}\label{cor.ineq3-c3}
\upsilon_{k+1}-\upsilon_k\leq2\Big[\Big(\frac{2}{n}\frac{1}{k}\sum_{i=1}^k\upsilon_i\Big)^2
-\Big(1+\frac{4}{n}\Big)\frac{1}{k}\sum_{j=1}^k\Big(\upsilon_j-\frac{1}{k}\sum_{i=1}^k\upsilon_i\Big)^2\Big]^\frac{1}{2}.
\end{equation}
\end{theorem}

Estimate~\eqref{cor.ineq1-c3} is a second-Yang type inequality due to Chen and Cheng~\cite[Inequality~(1.5)]{chen-cheng}, whereas the other two estimates generalize results by Cheng and Yang~\cite{Cheng-Yang-II,Cheng-Yang-I} obtained for a domain in the unit sphere $\mathbb{S}^n(1)$ and for a domain in the projective space $\mathbb{CP}^n(4)$. More precisely, inequalities \eqref{cor.ineq2-c3} and \eqref{cor.ineq3-c3} generalize Theorem~1 and Corollary~1 in \cite{Cheng-Yang-II}, respectively, as well as Theorem~1 and Corollary~1 in \cite{Cheng-Yang-I}.

\section{Motivating the definition of the operator $\mathscr{L}$}\label{Sec-L}
In this section, we establish the necessary tools to work with the operator defined in \eqref{defLintro} which enable us to obtain more general results. We believe this operator would be also useful in obtaining rigidity results or characterizing some known geometric objects.

Given an $n$-dimensional Riemannian manifold $(M,\langle,\rangle)$, to each $X\in\mathfrak{X}(M)$ we associate the $(0,1)$--tensor $X^\flat:\mathfrak{X}(M)\to C^{\infty}(M)$, given by
\begin{equation*}
X^\flat(Y)=\langle X,Y\rangle\quad \mbox{for all}\quad Y\in\mathfrak{X}(M).
\end{equation*}

Let $\,^{\sharp}\!:\mathfrak{X}^*(M)\to \mathfrak{X}(M)$ be the musical isomorphism, i.e., the inverse of the canonical mapping $\,^{\flat}\!:\mathfrak{X}(M)\to\mathfrak{X}^*(M)$.

Throughout the paper, we will be constantly using the identification of a $(0,2)$--tensor $T$ with its associated $(1,1)$--tensor by the equation
\begin{equation*}\label{AAAC-c3}
\langle T(X),Y\rangle=T(X,Y).
\end{equation*}
In particular, the metric tensor $\langle,\rangle$ will be identified with the identity $I$  in $\mathfrak{X}(M).$

Let $\{e_1,\ldots,e_n\}$ be an orthonormal basis in $T_pM$, and $S$ be a $(1,1)$--tensor with adjoint $S^*$. Recall that the {\it Hilbert-Schmidt inner product} is given by
\begin{equation*}
\langle T,S\rangle:=\mathrm{tr}(TS^*) = \sum_{i}\langle TS^*(e_i),e_i\rangle=\sum_{i}\langle S^{*}(e_i),T^*(e_i)\rangle= \sum_i\langle T(e_i),S(e_i)\rangle.
\end{equation*}

The divergence of a $(1,1)$--tensor $T$ in $(M,\langle,\rangle)$ is defined as the $(0,1)$--tensor
\begin{equation*}
(\dv T)(v)(p)=\mathrm{tr}(w\mapsto(\nabla_{w}T)(v)(p)),
\end{equation*}
where $p\in M$, $v,w\in T_{p}M,$ $\nabla$ stands for the covariant derivative of $T$ and $\mathrm{tr}$ is the trace operator calculated in the metric $\langle,\rangle.$

Furthermore, the Riemannian metric on $M$ induces the metric
\begin{equation*}
\langle X^\flat,Y^\flat\rangle=\langle X,Y\rangle,
\end{equation*}
where $X,Y\in\mathfrak{X}(M)$ and $X^\flat,Y^\flat\in\mathfrak{X}^*(M)$. It is easily verified that
\begin{equation*}
\langle\dv T,Z^\flat\rangle=\langle(\dv T)^\sharp,Z\rangle = (\dv T)(Z).
\end{equation*}

When there is no danger of confusion we must omit the symbol `` $^{\sharp}$ '' for the sake of simplicity.

For each $X\in\mathfrak{X}(M)$ we can consider the $(1,1)$--tensor $\nabla X(Y)=\nabla_YX$, for all $Y\in\mathfrak{X}(M)$. In this way, the divergence of $X$ is given by $\dv X=\mathrm{tr}(\nabla X).$ Let us define the $\eta$--divergence of $X$ as follows
\begin{equation*}
\dv_{\eta}X:=\dv X-\langle\nabla\eta,X\rangle.
\end{equation*}

From the usual properties of divergence of vector fields, one has
\begin{equation*}
\dv_{\eta}(fX)=f \dv_{\eta}X+\langle\nabla f,X\rangle\quad\mbox{and}\quad \dv(e^{-\eta}X)=e^{-\eta}\dv_{\eta}X
\end{equation*}
for all $f\in \mathcal{C}^{\infty}(M)$.

Suppose that $(M,\langle,\rangle)$ is an oriented Riemannian manifold, and let $\dM$ denote its Riemannian volume form. We will use the weighted measures $\dm=e^{-\eta}\dM$ and $\dn=e^{-\eta}\mathrm{d}\partial M$. If $\nu$ is the outward normal vector field on the boundary $\partial M$ and $X$ is a tangent vector field with compact support on $M$, then
\begin{equation*}
\int_{M}\dv_{\eta}X\dm=\int_{M}e^{-\eta}\dv_{\eta}X\mathrm{dM} =\int_{M}\dv(e^{-\eta}X)\dM =\int_{\partial M}\langle X,\nu\rangle\dn
\end{equation*}
which is the expression of the divergence theorem (or Stokes theorem) for the \emph{weighted manifold} $(M,\dm)$. Note that the drifting Laplacian $L$ is given by
\begin{equation*}
Lf=\dv_{\eta}\nabla f.
\end{equation*}
It is immediate that $L$ satisfies analogous properties to those of the Laplacian. For instance, for $f,\ell\in \mathcal{C}^{\infty}(M)$ we have
\begin{equation*}\label{F5-c3}
L(f\ell)=fL\ell+\ell Lf+2\langle\nabla f,\nabla\ell\rangle
\end{equation*}
and an extension of the well-known Bochner formula
\begin{equation*}
\frac{1}{2}L|\nabla f|^2= Ric_{\eta}(\nabla f,\nabla f)+|\nabla^2f|^2+\langle\nabla Lf,\nabla f\rangle,
\end{equation*}
where $Ric_{\eta}:=Ric+\nabla^{2}\eta$ is called the Bakry-Emery-Ricci tensor. This tensor has been especially studied in the theory of Ricci solitons and almost Ricci solitons, since a gradient Ricci soliton $(M,\langle,\rangle, \eta)$ is characterized by $Ric_{\eta}=\lambda\langle,\rangle$ for some constant $\lambda$, whereas in the case of almost gradient Ricci soliton, $\lambda$ is a variable function in $M$.

Now, let us make brief comment on specific points of the operator $\mathscr{L}$. It appears as the trace of the $(1,1)$--tensor on $M^n$, given by
\begin{equation*}
\tau_{\eta,f}:=\nabla T(\nabla f) - \frac{T(\nabla f,\nabla\eta)}{n}I.
\end{equation*}
Associated with $\tau_{\eta,f}$ one has
\begin{equation*}
\mathring{\tau_{\eta,f}}= \tau_{\eta,f}-\frac{\mathscr{L}f}{n}I  \quad\mbox{and}\quad|\tau_{\eta,f}|^2\geq \frac{1}{n}(\mathscr{L}f)^2.
\end{equation*}
Moreover, it has an $(\eta,T)$--divergence form, since $\mathscr{L}f:=\dv_{\eta}(T(\nabla f))$. Then, it is immediate from the properties of $\dv_{\eta}$ and the symmetry of $T$ that
\begin{equation*}\label{F5-c3}
\mathscr{L}(f\ell)=f\mathscr{L}\ell+\ell \mathscr{L}f+2T(\nabla f,\nabla\ell).
\end{equation*}

Alencar, Neto and Zhou~\cite{AGD} have recently proved a new Bochner-type formula and applied it to the operator which was introduced by Cheng and Yau~\cite{Cheng-Yau}, namely $$\Box f=\mathrm{tr}(\nabla^2f\circ T)=\langle\nabla^2f,T\rangle,$$ where $f\in \mathcal{C}^\infty(M)$ and $T$ is a symmetric $(1,1)$--tensor. Such formula is given by
\begin{equation}\label{BFpQ}
\frac{1}{2}\Box(|\nabla f|^2)=\langle\nabla(\Box f),\nabla f\rangle+R_T(\nabla f,\nabla f)+ \langle\nabla^2f,\nabla^2f\circ T\rangle-\langle\nabla^2f,\nabla_{\nabla f}T\rangle
\end{equation}
where $R_T(X,Y)=\mathrm{tr}(Z\mapsto T\circ R(X,Z)Y)$ and $R(X,Z)Y$  is the curvature tensor of $(M,\langle,\rangle)$.

A straightforward computation gives us
\begin{equation}\label{C7-c3}
\dv_{\eta}(T(\ell\nabla f))=\ell\langle\dv_{\eta}T,\nabla f\rangle+\ell\langle\nabla^{2}f,T\rangle+T(\nabla \ell,\nabla f),
\end{equation}
where $\dv_{\eta}T:=\dv T-\mathrm{d}\eta\circ T$ is the \emph{$\eta$--divergence} of a symmetric tensor $T$, and $\mathrm{d}\eta\circ T=\langle\nabla\eta,T(\cdot)\rangle=T(\nabla\eta,\cdot)$. So, we get the following equation
\begin{equation}\label{rel_L_dv_Box-c3}
\mathscr{L} f=\langle\dv_{\eta}T,\nabla f\rangle+\Box f.
\end{equation}

Equation \eqref{rel_L_dv_Box-c3} relates the operators $\mathscr{L}$, $\dv_\eta T$ and $\Box$. It also gives us a \emph{Bochner-type formula for the operator $\mathscr{L}$} as follows:
\begin{equation*}
\frac{1}{2}\mathscr{L}(|\nabla f|^2)=\langle\nabla(\mathscr{L}f),\nabla f\rangle+ R_{\eta,T}(\nabla f,\nabla f) +\langle\nabla^2f,\nabla^2f\circ T\rangle-\langle\nabla^2f,\nabla_{\nabla f}T\rangle
\end{equation*}
where $R_{\eta,T}:=R_T-\nabla(\dv_\eta T)^\sharp$. Indeed, from equation \eqref{rel_L_dv_Box-c3} we deduce that
\begin{align*}
\langle \nabla(\mathscr{L}f),\nabla f\rangle =&\langle \nabla_{\nabla f}\dv_{\eta}T,\nabla f \rangle + \langle\dv_{\eta}T,\nabla_{\nabla f}\nabla f \rangle + \langle \nabla(\Box f),\nabla f \rangle\\
=& \nabla (\dv_{\eta}T)^{\sharp}(\nabla f,\nabla f) +\frac{1}{2}\langle\dv_{\eta}T,\nabla(|\nabla f|^2)\rangle + \langle \nabla(\Box f),\nabla f \rangle\\
=& \nabla (\dv_{\eta}T)^{\sharp}(\nabla f,\nabla f) +\frac{1}{2}\mathscr{L}(|\nabla f|^2) - \frac{1}{2}\Box(|\nabla f|^2) + \langle \nabla(\Box f),\nabla f \rangle,
\end{align*}
that is,
\begin{equation*}
\frac{1}{2}\mathscr{L}(|\nabla f|^2) =\langle \nabla(\mathscr{L}f),\nabla f\rangle - \nabla (\dv_{\eta}T)^{\sharp}(\nabla f,\nabla f)+\frac{1}{2}\Box(|\nabla f|^2) - \langle \nabla(\Box f),\nabla f \rangle.
\end{equation*}
Therefore, our assertion follows from Bochner-type formula \eqref{BFpQ}.

Having proven the Bochner-type formula for $\mathscr{L}$, we can study hypersurfaces with constant mean curvature $k$ by means of the properties of the operator $\mathscr{L}$ as well as to obtain some generalizations in the spirit of Lichnerowicz's and Obata's theorems. Using standard techniques we can even hope to generalize the recent results obtained by Alencar, Neto and Zhou \cite{AGD}. Furthermore, we can also obtain some generalizations of the work of Uhlenbeck \cite{Uhlenbeck}, Berger~\cite{berger} and El Soufi and Ilias~\cite{ahmad}.

For orientable compact Riemannian manifolds, Cheng and Yau proved that the operator $\Box$ is self-adjoint if and only if $\dv T=0$. Indeed, from the divergence theorem and equation~\eqref{C7-c3} we get
\begin{equation*}
\int_M(f\Box \ell-\ell\Box f)\dM=\int_M \big(\ell\langle \dv T,\nabla f\rangle - f\langle\dv T,\nabla \ell\rangle\big)\dM.
\end{equation*}
Hence, the operator $\Box$ is self-adjoint if and only if $\dv T=0$. In this case, equation \eqref{rel_L_dv_Box-c3} reduces to
\begin{equation*}
\mathscr{L} f=-T(\nabla\eta,\nabla f)+\Box f.
\end{equation*}

For instance, denoting $R=\mathrm{tr}(Ric)$, it is well-known that $\dv Ric = \frac{\mathrm{d}R}{2}$ and $\dv(RI)=\mathrm{d}R$, so the Einstein tensor $G:=Ric-\frac{R}{2}\langle,\rangle$ has null divergence, and therefore $\Box f=\langle\nabla^2f,G\rangle$ is self-adjoint. Then,
\begin{equation*}
\mathscr{L} f=-G(\nabla\eta,\nabla f)+\langle\nabla^2f,G\rangle.
\end{equation*}
This particular case is likely to have applications in physics.

\section{The Dirichlet Problem}
In what follows $(M,\langle,\rangle)$ is a complete Riemannian manifold and the domain $\Omega\subset M$ is connected bounded with smooth boundary $\partial\Omega$ in $M$. Notice that the divergence theorem remains true in the form
\begin{equation*}
\int_{\Omega}\mathscr{L}f\dm=\int_{\partial\Omega}T(\nabla f,\nu)\dn,
\end{equation*}
where $\dn=e^{-\eta}\mathrm{d}\partial\Omega$ is the weight volume form on the boundary induced by the outward normal vector $\nu$ on $\partial\Omega$. Thus, the ``integration by parts'' formula is
\begin{equation*}
\int_{\Omega}\ell\mathscr{L}f\dm=-\int_{\Omega}T(\nabla\ell,\nabla f)\dm+\int_{\partial\Omega}\ell T(\nabla f,\nu)\dn.
\end{equation*}

Therefore, $\mathscr{L}$ is a self-adjoint operator in the space of all functions in $L^2(\Omega,\dm)$ that vanish on $\partial\Omega$. Thus the eigenvalue problem
\begin{equation}\label{PD-c3}
\left\{
\begin{array}{ccccc}
-\mathscr{L}u &=& \lambda u & \hbox{in} & \Omega,\\
u&=&0 & \hbox{on} & \partial\Omega
\end{array}
\right.
\end{equation}
has a real and discrete spectrum $0<\lambda_{1}\leq\lambda_{2}\ldots\leq\lambda_{k}\ldots\rightarrow \infty$,  where each $\lambda_i$
is repeated according to its multiplicity. In particular, for an eigenfunctions $u_i$ we have
\begin{eqnarray*}
\lambda_i=-\int_{\Omega}u_i\mathscr{L}u_i\dm=\int_{\Omega}T(\nabla
u_i,\nabla u_i)\dm.
\end{eqnarray*}

Operators in divergence form have been studied in different settings (see, e.g., \cite{AJuR,Li Ma,Setti,Zeng}). In particular, the drifting Laplacian $L$ is closely related to problems in spaces with warped products or Ricci solitons (see, e.g., \cite{CM04,RSWP,Ma06}). For instance, Zeng~\cite{Zeng} obtained some eigenvalue inequalities of $L$ on the Ricci solitons with certain conditions. More recently, by using of the strong maximum principle for $L$, it has been proved that an expanding or steady gradient Ricci soliton warped product $B^n\times_f F^m$, $m>1$, whose warping function $f$ reaches both maximum and minimum must be a Riemannian product \cite{RSWP}. The other example is when $\Omega\subset\Bbb{R}^n$, in this case, one considers the warped metric $g=g_0+e^{-\eta}d\theta^2$ on the product $\Omega\times\mathbb{S}^1$, where $g_0$ stands for the Riemannian metric in the domain $\Omega$, whereas $d\theta^2$ is the canonical metric of $\Bbb{S}^1$, then the scalar curvature of the metric $g$ is given by $f=\frac{1}{4}(2\Delta\eta -|\nabla\eta|^2)$. We mention that the modified scalar curvature of a metric $g$ and  a dilatation function $\eta$, as introduced by Perelman, is $\mathrm{R}^m=\mathrm{R}+4f$, where $\mathrm{R}$ is the scalar curvature of $g$.

Eigenvalues of differential operators find their applications in many areas. For example, in quantum mechanics  quantities like energy, momentum and position are represented by hermitian operators in a Hilbert space. The eigenvalues of the operator that corresponds to the energy are interpreted as the possible values of energy that the system can attain. In addition, the gap between them is simply the gap between the energy levels. A well-known mathematical question that motivates to study the properties of the spectrum of operators is about domains isospectral, for instance, if two domains are isospectral, is it necessarily true that they are isometric?

Weyl~\cite{Weyl} studied Problem~\eqref{PD-c3} for the Laplacian case, being the first to publish a proof of the asymptotic behavior of the eigenvalues of the Laplacian on domains in $\Bbb{R}^n$. However, Rayleigh had already studied such behavior few years earlier which had actually been conjectured by Sommerfeld and Lorentz. The Weyl's asymptotic formula is given by
\begin{equation}\label{introd-weyl-c3}
\lambda_k\sim\frac{4\pi^2}{(\omega_n \mathrm{vol}\,\Omega)^\frac{2}{n}}k^{\frac{2}{n}} \quad (k\to\infty)
\end{equation}
where $\omega_n$ is the volume of the unit ball in $\mathbb{R}^n$. The discovery of Weyl has become an important tool for comparison  of estimates related to eigenvalues, not only for the Laplacian, but also for the more general elliptic operators. Furthermore, the domains in $\Bbb{R}^n$ could be replaced by domains in Riemannian manifolds.

Good estimates are those that provide the best bounds taking into account the Weyl's asymptotic formula \eqref{introd-weyl-c3}. For example, if $\Omega$ is a tiling domain in $\Bbb{R}^n$ and $\lambda_k$ is a Laplace eigenvalue of the Dirichlet eigenvalue problem, then P\'olya~\cite{Polya} proved the inequality
\begin{equation}\label{Eq-polya}
\lambda_k\geq\frac{4\pi^2}{(\omega_n \mathrm{vol}\,\Omega)^\frac{2}{n}}k^{\frac{2}{n}}
\end{equation}
for $k\in\{1,2,\ldots\}$. He conjectured that this inequality is valid for a general domain in $\Bbb{R}^n$. In relation to this conjecture, Li and Yau \cite{Li-Yau} proved to be valid the inequality
\begin{equation}\label{partial-Yau}
\frac{1}{k}\sum_{i=1}^k\lambda_i\geq\frac{n}{n+2}\frac{4\pi^2}{(\omega_n \mathrm{vol}\,\Omega)^\frac{2}{n}}k^{\frac{2}{n}}.
\end{equation}

Cheng and Yang established an analog of the Li and Yau's inequality for Laplace eigenvalues of the Dirichlet eigenvalue problem on a domain in a complete Riemannian manifold, see \cite[Theorem~1.1]{Cheng-Yang-III}. From their results, they proposed a generalization of the P\'olya's conjecture, namely:

\begin{conjecture}\label{conjecture}
Let $\Omega$ be a domain in an $n$-dimensional complete Riemannian manifold $M$. Then, there exists a constant $c(M,\Omega)$, which only depends on $M$ and $\Omega$ such that Laplace eigenvalues $\lambda_i$'s of the Dirichlet eigenvalue problem, satisfy
\begin{eqnarray*}
\frac{1}{k}\sum_{i=1}^k\lambda_i + c(M,\Omega) &\geq& \frac{n}{n+2}\frac{4\pi^2}{(\omega_n \mathrm{vol}\,\Omega)^\frac{2}{n}}k^{\frac{2}{n}}, \quad\mbox{for}\quad k=1,2,\ldots\\
\lambda_k + c(M,\Omega) &\geq& \frac{4\pi^2}{(\omega_n \mathrm{vol}\,\Omega)^\frac{2}{n}}k^{\frac{2}{n}}, \quad\mbox{for}\quad k=1,2,\ldots
\end{eqnarray*}
\end{conjecture}

We now start with the proofs of the main results of the present paper. First we prove the proposition below which is motivated by the corresponding results for the Laplacian case proven in Cheng and Yang~\cite[Proposition~1]{Cheng-Yang-I} and for the drifting Laplacian case proven in Xia and Xu~\cite[Theorem~1.1]{Xia-Xu}. Here, we follow the steps of the proof of Theorem~1.1 in \cite{Xia-Xu} with appropriate adaptations for $\mathscr{L}$.

\begin{proposition}\label{Lem1T-c3}
Let $\lambda_i$ be the $i$-th eigenvalue of Problem \eqref{PD-c3} and $u_i$ its corresponding normalized real-valued eigenfunction. Then, for $h \in \mathcal{C}^3(\Omega)\cap\mathcal{C}^2(\partial \Omega)$, and $k$ integer, is valid
\begin{equation*}
\sum_{i=1}^k(\lambda_{k+1}-\lambda_i)^2\int_{\Omega} u_i^2T(\nabla h,\nabla h)\dm \leq\sum_{i=1}^k(\lambda_{k+1} -\lambda_i)\int_{\Omega}(u_i\mathscr{L}h+2T(\nabla h,\nabla u_i))^2\dm.
\end{equation*}
\end{proposition}
\begin{proof}
By the inequality of Rayleigh-Ritz, we have
\begin{equation}\label{IneqRR}
\lambda_{k+1}\leq -\frac{\int_{\Omega}\psi \mathscr{L}\psi\dm}{\int_{\Omega}\psi^2\dm},
\end{equation}
for any no null function $\psi:\Omega\rightarrow\mathbb{R}$ satisfying
\begin{equation*}
\psi|_{\partial \Omega}=0  \quad\mbox{and}\quad  \int_{\Omega}\psi u_i \dm=0, \quad\forall i=1,\ldots,k.
\end{equation*}
For $h \in\mathcal{C}^3(\Omega)\cap\mathcal{C}^2(\partial\Omega)$ and for each $i$, we set
\begin{equation}\label{phi-i}
\phi_i=hu_i - \sum_{j=1}^{k}c_{ij}u_j
\end{equation}
where $c_{ij}= \int_{\Omega} hu_iu_j\dm$, so that, $\phi_i|_{\partial\Omega}=0$ and
\begin{equation*}
0=\int_{\Omega} \phi_i u_l \dm=\int_{\Omega}hu_iu_l \dm-\sum_{j=1}^{k}c_{ij}\int_{\Omega} u_ju_l\dm=\int_{\Omega}hu_iu_l \dm -\sum_{j=1}^{k}c_{ij}\delta_{jl}.
\end{equation*}
Then, we can make $\psi=\phi_i$ in \eqref{IneqRR}, i.e.,
\begin{equation}\label{RRfi-c3}
\lambda_{k+1}\leq - \frac{\int_{\Omega}\phi_i \mathscr{L} \phi_i \dm}{\int_{\Omega}\phi_i^2 \dm}.
\end{equation}
Since
\begin{equation*}
\mathscr{L}\phi_i = \mathscr{L}(hu_i)-\sum_{j=1}^{k}c_{ij}\mathscr{L}u_j=-\lambda_i hu_i +u_i\mathscr{L}h+2T(\nabla h,\nabla u_i) +\sum_{j=1}^{k}c_{ij}\lambda_ju_j,
\end{equation*}
from \eqref{phi-i} we have
\begin{align}\label{-fiLfi-c3}
\nonumber\phi_i\mathscr{L}\phi_i&=\phi_i(u_i\mathscr{L}h+2T(\nabla h,\nabla u_i))-\phi_i\lambda_i hu_i+\sum_{j=1}^{k}c_{ij}\lambda_ju_j\phi_i\\
&=\phi_i(u_i\mathscr{L}h+2T(\nabla h,\nabla u_i))-\lambda_i\phi_i^2-\sum_{j=1}^{k}c_{ij}\lambda_iu_j\phi_i +\sum_{j=1}^{k}c_{ij}\lambda_ju_j\phi_i.
\end{align}
Substituting \eqref{-fiLfi-c3} into \eqref{RRfi-c3}, we deduce that
\begin{equation*}
\lambda_{k+1}\leq-\frac{\int_{\Omega}\big(\phi_i(u_i\mathscr{L}h+2T(\nabla h,\nabla u_i))-\lambda_i\phi_i^2\big)\dm}{\int_{\Omega}\phi_i^2\dm}.
\end{equation*}
Hence
\begin{equation}\label{1.0-c3}
(\lambda_{k+1}-\lambda_i)|\phi_i|^2\leq-\int_{\Omega}\phi_i\big(u_i\mathscr{L}h+2T(\nabla h,\nabla u_i)\big)\dm,
\end{equation}
where $|\phi_i|^2=|\phi_i|_{L^2(\Omega,\dm)}^2$. We now estimate
\begin{equation*}
P_i=- \int_{\Omega}\phi_i\big(u_i\mathscr{L}h+2T(\nabla h,\nabla u_i)\big)\dm.
\end{equation*}
For this, we set
\begin{equation*}
b_{ij}=-\int(u_j\mathscr{L}h+2T(\nabla h,\nabla u_j))u_i\dm.
\end{equation*}
We use that $\int_{\Omega}\phi_iu_j\dm=0$ again to get
\begin{eqnarray*}
P_i&=&-\int_{\Omega}\phi_i\Big(u_i\mathscr{L}h+2T(\nabla h,\nabla u_i)-\sum_{j=1}^{k} b_{ij}u_j\Big)\dm\\
&\leq&|\phi_i||u_i\mathscr{L}h+2T(\nabla h,\nabla u_i)-\sum_{j=1}^{k} b_{ij}u_j|.
\end{eqnarray*}
Therefore
\begin{equation}\label{Pi-c3}
(\lambda_{k+1}-\lambda_i)^2P_i^2\leq(\lambda_{k+1}-\lambda_i)^2 |\phi_i|^2|u_i\mathscr{L}h+2T(\nabla h,\nabla u_i)- \sum_{j=1}^{k} b_{ij}u_j|^2.
\end{equation}
From \eqref{1.0-c3}, $(\lambda_{k+1}-\lambda_i) |\phi_i|^2\leq P_i$, then \eqref{Pi-c3} implies
\begin{equation*}
(\lambda_{k+1}-\lambda_i)^2P_i \leq (\lambda_{k+1}-\lambda_i)|u_i\mathscr{L}h+2T(\nabla h,\nabla u_i)- \sum_{j=1}^{k} b_{ij}u_j|^2,
\end{equation*}
whence
\begin{equation}\label{1.1 dif-c3}
\sum_{i=1}^{k}(\lambda_{k+1}-\lambda_i)^2P_i \leq\sum_{i=1}^{k}(\lambda_{k+1}-\lambda_i)|u_i\mathscr{L}h+2T(\nabla h,\nabla u_i) - \sum_{j=1}^{k} b_{ij}u_j|^2.
\end{equation}
Moreover,
\begin{eqnarray*}
\lambda_i c_{ij} & = &- \int_{\Omega}hu_j\mathscr{L}u_i \dm = -\int_{\Omega}u_i\mathscr{L}(hu_j) \dm\\
&=&-\int_{\Omega}u_i\big(-\lambda_jhu_j+u_j\mathscr{L}h+2T(\nabla h,\nabla u_j)\big)\dm = \lambda_jc_{ij}+b_{ij},
\end{eqnarray*}
thus $b_{ij}=(\lambda_i-\lambda_j)c_{ij}$ and $b_{ij}=-b_{ji}$. Besides,
\begin{eqnarray}\label{1.2-c3}
\nonumber &&|u_i\mathscr{L}h+2T(\nabla h,\nabla u_i)-\sum_{j=1}^{k} b_{ij}u_j|^2\\
\nonumber&=&|u_i\mathscr{L}h+2T(\nabla h,\nabla u_i)|^2+ \sum_{j=1}^{k} b_{ij}^2 -2\sum_{j=1}^{k} b_{ij}\int_{\Omega}\big(u_i\mathscr{L}h+2T(\nabla h,\nabla u_i)\big)u_j\dm\\
\nonumber&=&|u_i\mathscr{L}h+2T(\nabla h,\nabla u_i)|^2-\sum_{j=1}^{k} b_{ij}^2\\
&=&|(u_i\mathscr{L}h+2T(\nabla h,\nabla u_i)|^2 - \sum_{j=1}^{k}(\lambda_i-\lambda_j)^2c_{ij}^2
\end{eqnarray}
and
\begin{eqnarray}\label{1.3-c3}
\nonumber P_i&=&-\int_{\Omega}(hu_i - \sum_{j=1}^{k} c_{ij}u_j)(u_i\mathscr{L}h+2T(\nabla h,\nabla u_i))\dm\\
\nonumber&=&-\int_{\Omega}\big(hu_i^2 \mathscr{L}h +2hu_iT(\nabla h,\nabla u_i)\big)\dm +\sum_{j=1}^{k} c_{ij}b_{ij}\\
\nonumber&=&\int_{\Omega}T(\nabla (hu_i^2),\nabla h)\dm-2\int_{\Omega}hu_iT(\nabla h,\nabla u_i)\dm +\sum_{j=1}^{k}(\lambda_i-\lambda_j)c_{ij}^2\\
&=&\int_{\Omega} u_i^2T(\nabla h,\nabla h)\dm +\sum_{j=1}^{k}(\lambda_i-\lambda_j)c_{ij}^2.
\end{eqnarray}
Substituting \eqref{1.2-c3} and \eqref{1.3-c3} into \eqref{1.1 dif-c3} we obtain
\begin{align}\label{eqFAuf1}
&\sum_{i=1}^{k}(\lambda_{k+1}-\lambda_i)^2\int_{\Omega} u_i^2T(\nabla h,\nabla h)\dm +\sum_{i,j=1}^{k}(\lambda_{k+1} -\lambda_i)^2(\lambda_i-\lambda_j)c_{ij}^2\\
\nonumber\leq&\sum_{i=1}^{k}(\lambda_{k+1}-\lambda_i)\int_{\Omega}(u_i\mathscr{L}h+2T(\nabla h,\nabla u_i))^2\dm-\sum_{i,j=1}^{k}(\lambda_{k+1}
-\lambda_i)(\lambda_i-\lambda_j)^2c_{ij}^2.
\end{align}
We observe that
\begin{align}\label{eqFAuf2}
\nonumber&\sum_{i,j=1}^{k}(\lambda_{k+1}-\lambda_i)^2(\lambda_i-\lambda_j)c_{ij}^2=\sum_{i,j=1}^{k}(\lambda_{k+1}-\lambda_i)
(\lambda_{k+1}-\lambda_j+\lambda_j-\lambda_i)(\lambda_i-\lambda_j)c_{ij}^2\\
\nonumber&=\sum_{i,j=1}^{k}(\lambda_{k+1}-\lambda_i)(\lambda_{k+1}-\lambda_j)(\lambda_i-\lambda_j)c_{ij}^2+\sum_{i,j=1}^{k}
(\lambda_{k+1}-\lambda_i)(\lambda_j-\lambda_i)(\lambda_i-\lambda_j)c_{ij}^2\\
&=-\sum_{i,j=1}^{k}(\lambda_{k+1}-\lambda_i)(\lambda_i-\lambda_j)^2c_{ij}^2,
\end{align}
where we have used the fact that $\sum_{i,j=1}^{k}(\lambda_{k+1}-\lambda_i)(\lambda_{k+1}-\lambda_j)(\lambda_i-\lambda_j)c_{ij}^2=0$. Finally, the required inequality in the proposition follows immediately from \eqref{eqFAuf1} and \eqref{eqFAuf2}.
\end{proof}

The following result is an extension for the operator $\mathscr{L}$ of the inequality obtained for the Laplacian case in Chen-Cheng \cite[Theorem~1.1]{chen-cheng}. We follow the steps of the proof of Theorem~1.2 in Xia and Xu~\cite{Xia-Xu} with appropriate adaptations for $\mathscr{L}$. In Section~\ref{sec-eta-laplaciano}, we can notice that this is a good extension for $\mathscr{L}$ of Theorem~1.2 in~\cite{Xia-Xu}.

\begin{proposition}\label{lemmaCYT-c3}
Let $\Omega$ be a domain of an $n$-dimensional complete Riemannian manifold $M$ isometrically immersed in $\mathbb{R}^m$, $\lambda_i$ be the $i$-th eigenvalue of Problem \eqref{PD-c3} and $u_i$ its corresponding normalized real-valued eigenfunction. Then
\begin{eqnarray*}
\mathrm{tr}(T)\sum_{i=1}^k(\Lambda_i^1)^2&\leq&\sum_{i=1}^k\Lambda_i^1\Big[\int_{\Omega}u_i^2\Big(\|\mathrm{tr}(\alpha\circ T)\|^2+|\mathrm{tr}(\nabla T)-T(\nabla\eta)|^2\Big)\dm\\
&&+4\int_{\Omega}u_i\Big(\langle\mathrm{tr}(\nabla T),T(\nabla u_i)\rangle-\langle T(\nabla\eta),T(\nabla u_i)\rangle\Big)\dm+4\lambda_i\Big],
\end{eqnarray*}
where $\Lambda_i^1=\lambda_{k+1}-\lambda_i$, $\alpha$ is the second fundamental form of $M$ and $\alpha\circ T=\alpha(T(\cdot),\cdot)$.
\end{proposition}
\begin{proof}
Let $x=(x_1,\ldots,x_m)$ be the position vector of the immersion of $M$ in $\mathbb{R}^m$. Consider $h=x_{\ell}$ in Proposition~\ref{Lem1T-c3} and take the sum in $\ell$ to get
\begin{eqnarray*}
&&\sum_{i=1}^k(\Lambda_i^1)^2\int_{\Omega}u_i^2\sum_{\ell=1}^mT(\nabla x_{\ell},\nabla x_{\ell})\dm\\
&\leq&\sum_{i=1}^k\Lambda_i^1\int_{\Omega}\sum_{\ell=1}^m(u_i\mathscr{L}x_{\ell}+2T(\nabla x_{\ell},\nabla u_i))^2\dm\\
&=&\sum_{i=1}^k\Lambda_i^1\int_{\Omega}\sum_{\ell=1}^m\Big(u_i\mathrm{div}_{\eta}(T(\nabla x_{\ell}))+2T(\nabla x_{\ell},\nabla u_i)\Big)^2\dm.
\end{eqnarray*}
So,
\begin{eqnarray}\label{01T-c3}
\nonumber &&\sum_{i=1}^k(\Lambda_i^1)^2 \int_{\Omega} u_i^2\sum_{\ell=1}^mT(\nabla x_{\ell},\nabla x_{\ell})\dm\\
\nonumber&\leq&\sum_{i=1}^k \nonumber\Lambda_i^1\int_{\Omega}\sum_{\ell=1}^m\Big(u_i^2(\mathrm{div}_{\eta}(T(\nabla x_{\ell}))^2+4u_i\mathrm{div}_{\eta}(T(\nabla x_{\ell}))T(\nabla x_{\ell},\nabla u_i)\\
&&+4T(\nabla x_{\ell},\nabla u_i)^2\Big)\dm.
\end{eqnarray}
Denoting  the canonical connection of $\mathbb{R}^m$ by $\nabla^0$ and taking $\{e_1,\ldots, e_m\}$ a local orthonormal geodesic frame in $p\in M$ adapted to $M$, we can write
\begin{eqnarray*}
\nabla^0 x_{\ell}&=&\sum_{i=1}^n e_i(x_\ell)e_i+\sum_{i=n+1}^m e_i(x_\ell)e_i,\\
e_{\ell}&=&\nabla x_{\ell}+(\nabla x_{\ell})^\perp.
\end{eqnarray*}
Thus
\begin{eqnarray}\label{03T-c3}
\nonumber\sum_{\ell=1}^m T(\nabla x_{\ell},\nabla u_i)^2&=&\sum_{\ell=1}^m \langle\nabla x_{\ell},T(\nabla u_i)\rangle^2=\sum_{\ell=1}^m\langle e_{\ell}-(\nabla x_{\ell})^\perp,T(\nabla u_i)\rangle^2\\
&=&\sum_{\ell=1}^n\langle e_{\ell},T(\nabla u_i)\rangle^2=|T(\nabla u_i)|^2,
\end{eqnarray}
and
\begin{eqnarray}\label{02T-c3}
\nonumber\sum_{\ell=1}^m T(\nabla x_{\ell},\nabla x_{\ell})&=&\sum_{\ell=1}^m\langle\nabla x_{\ell},T(\nabla x_{\ell})\rangle=\sum_{\ell=1}^m\langle e_{\ell}-(\nabla x_{\ell})^\perp,T(\nabla x_{\ell})\rangle\\
\nonumber&=&\sum_{\ell=1}^n\langle e_{\ell},T(\nabla x_{\ell})\rangle=\sum_{\ell=1}^n\langle T(e_{\ell}),\nabla x_{\ell}\rangle\\
&=&\sum_{l=1}^n \langle e_{\ell},T(e_{\ell})\rangle=\mathrm{tr}(T).
\end{eqnarray}
Moreover,
\begin{align}\label{06T-c3}
\nonumber\mathrm{div}_{\eta}(T(\nabla x)):=&\big(\mathrm{div}_{\eta}(T(\nabla x_1)),\ldots,\mathrm{div}_{\eta}(T(\nabla x_m))\big)\\
\nonumber=&(\mathrm{div}(T(\nabla x_1))-\langle \nabla\eta, T(\nabla x_1)\rangle,\ldots,\mathrm{div}(T(\nabla x_m))-\langle \nabla\eta,T(\nabla x_m)\rangle)\\
=&\mathrm{div}(T(\nabla x))-\mathrm{d}\eta\circ T(\nabla x).
\end{align}
Next, we compute
\begin{eqnarray*}
\mathrm{div}(T(\nabla x))&:=&\Big(\mathrm{div}(T(\nabla x_1)),\ldots,\mathrm{div}(T(\nabla x_m))\Big)\\
&=&\Big(\sum_{i=1}^n e_i\langle T(\nabla x_1),e_i\rangle,\ldots,\sum_{i=1}^n e_i\langle T(\nabla x_m),e_i\rangle\Big)\\
&=&\sum_{i=1}^n\Big(e_i\langle T(\sum_{j=1}^n e_j(x_1)e_j),e_i\rangle,\ldots,e_i\langle T(\sum_{j=1}^n e_j(x_m)e_j),e_i\rangle\Big)\\
&=&\sum_{i,j=1}^n\Big(e_i\big(e_j(x_1)\langle T(e_j),e_i\rangle\big),\ldots,e_i\big(e_j(x_m)\langle T(e_j),e_i\rangle\big)\Big).
\end{eqnarray*}
Then,
\begin{eqnarray*}
\mathrm{div}(T(\nabla x))&=&\sum_{i,j=1}^n\Big(e_ie_j(x_1)\langle T(e_j),e_i\rangle,\ldots,e_ie_j(x_m)\langle T(e_j),e_i\rangle\Big)\\
&&+\sum_{i,j=1}^n\Big(e_j(x_1)\langle\nabla_{e_i}T(e_j),e_i\rangle,\ldots,e_j(x_m)\langle\nabla_{e_i}T(e_j),e_i\rangle\Big)\\
&=&\sum_{i,j=1}^n\langle T(e_j),e_i\rangle\Big(e_ie_j(x_1),\ldots,e_ie_j(x_m)\Big)\\
&&+\sum_{i,j=1}^n\langle\nabla_{e_i}T(e_j),e_i\rangle\Big(e_j(x_1),\ldots,e_j(x_m)\Big).
\end{eqnarray*}
Hence
\begin{eqnarray}\label{divT-c3}
\nonumber\mathrm{div}(T(\nabla x))&=&\sum_{i,j=1}^n\langle T(e_j),e_i\rangle(\nabla^0_{e_i}e_j)(x)+\sum_{i,j=1}^n\langle\nabla_{e_i}T(e_j),e_i\rangle e_j(x)\\
\nonumber&=&\sum_{i,j=1}^n\langle T(e_j),e_i\rangle\nabla^0_{e_i}e_j+\sum_{i,j=1}^n\langle\nabla_{e_i}T(e_j),e_i\rangle e_j\\
\nonumber&=&\sum_{i,j=1}^n\langle T(e_j),e_i\rangle(\nabla_{e_i}e_j+\alpha(e_i,e_j))+\sum_{i,j=1}^n\langle\nabla_{e_i}T(e_j),e_i\rangle e_j\\
\nonumber&=&\sum_{i,j=1}^n\langle T(e_j),e_i\rangle\alpha(e_i,e_j)+\sum_{i,j=1}^n\langle\nabla_{e_i}T(e_j),e_i\rangle e_j\\
&=&\sum_{j=1}^n\alpha(T(e_j),e_j)+\sum_{i,j=1}^n\langle\nabla_{e_i}T(e_j),e_i\rangle e_j.
\end{eqnarray}
Since T is symmetric we have $\langle T(e_j),e_i\rangle=\langle e_j,T(e_i)\rangle$, which implies
\begin{equation*}
\langle\nabla_{e_i}T(e_j),e_i\rangle=\langle e_j,\nabla_{e_i}T(e_i)\rangle.
\end{equation*}
Substituting this previous equation into \eqref{divT-c3}, we have
\begin{eqnarray}\label{2-Tr}
\nonumber\mathrm{div}(T(\nabla x))&=&\sum_{i=1}^n\alpha(T(e_i),e_i)+\sum_{i=1}^n\nabla_{e_i}T(e_i)\\
\nonumber&=&\sum_{i=1}^n\alpha(T(e_i),e_i)+\sum_{i=1}^n(\nabla_{e_i}T)(e_i)\\
\nonumber&=&\sum_{i=1}^n\alpha(T(e_i),e_i)+\sum_{i=1}^n(\nabla T)(e_i,e_i)\\
&=&\mathrm{tr}(\alpha(T(\cdot),\cdot))+\mathrm{tr}(\nabla T),
\end{eqnarray}
where
\begin{equation*}
\mathrm{tr}(\alpha(T(\cdot),\cdot)):=\sum_{i=1}^n\alpha(T(e_i),e_i)\in \mathfrak{X}(M)^\perp \,\,\,\mbox{and}\,\,\, \mathrm{tr}(\nabla T):=\sum_{i=1}^n(\nabla T)(e_i,e_i)\in \mathfrak{X}(M).
\end{equation*}
Besides,
\begin{eqnarray}\label{d_etaT-c3}
\nonumber \mathrm{d}\eta\circ T(\nabla x)&:=&\big(\langle\nabla\eta,T(\nabla x_1)\rangle,\ldots,\langle\nabla\eta,T(\nabla x_m)\rangle\big)\\
\nonumber&=&\big(\langle\nabla\eta,T(\sum_{i=1}^n e_i(x_1)e_i)\rangle,\ldots,\langle\nabla\eta,T(\sum_{i=1}^n e_i(x_m)e_i)\rangle\big)\\
\nonumber&=&\sum_{i=1}^n\langle\nabla\eta,T(e_i)\rangle( e_i(x_1),\ldots,e_i(x_m))\\
&=&\sum_{i=1}^n\langle\nabla\eta,T(e_i)\rangle e_i(x)=\sum_{i=1}^n\langle T(\nabla\eta),e_i\rangle e_i=T(\nabla\eta).
\end{eqnarray}
Substituting \eqref{2-Tr} and \eqref{d_etaT-c3} into \eqref{06T-c3}, we obtain
\begin{equation*}
\mathrm{div}_{\eta}(T(\nabla x))=\mathrm{tr}(\alpha(T(\cdot),\cdot))+\mathrm{tr}(\nabla T)-T(\nabla\eta).
\end{equation*}
We now compute
\begin{eqnarray}\label{09T-c3}
\nonumber\sum_{\ell=1}^m(\mathrm{div}_{\eta}(T(\nabla x_{\ell})))^2&=&\|\mathrm{div}_{\eta}(T(\nabla x))\|^2\\
&=&\|\mathrm{tr}(\alpha(T(\cdot),\cdot))\|^2+|\mathrm{tr}(\nabla T)-T(\nabla\eta)|^2
\end{eqnarray}
and
\begin{align}\label{08T-c3}
\nonumber\sum_{\ell=1}^m \mathrm{div}_{\eta}(T(\nabla x_{\ell}))T(\nabla x_{\ell},\nabla u_i)&=\sum_{\ell=1}^m \mathrm{div}_{\eta}(T(\nabla x_{\ell}))T(\nabla u_i)(x_{\ell})\\
\nonumber&=\langle\mathrm{div}_{\eta}(T(\nabla x)),T(\nabla u_i)\rangle\\
\nonumber&=\langle\mathrm{tr}(\nabla T)-T(\nabla\eta),T(\nabla u_i)\rangle\\
&=\langle\mathrm{tr}(\nabla T),T(\nabla u_i)\rangle-\langle T(\nabla\eta),T(\nabla u_i)\rangle.
\end{align}
Substituting \eqref{03T-c3}, \eqref{02T-c3}, \eqref{09T-c3} and \eqref{08T-c3} into \eqref{01T-c3} we complete the proof of the proposition.
\end{proof}

\subsection{Proof of Theorem~\ref{propCYT-c3}}
\begin{proof}
We make use of Proposition~\ref{lemmaCYT-c3}. For this, we begin by computing
\begin{eqnarray*}
\nonumber &&\|\mathrm{tr}(\alpha(T(\cdot),\cdot))\|^2+|\mathrm{tr}(\nabla T)-T(\nabla\eta)|^2\\
&=&\|\sum_{i=1}^n\alpha(T(e_i),e_i)\|^2+|\mathrm{tr}(\nabla T)|^2 -2\langle\mathrm{tr}(\nabla T),T(\nabla\eta)\rangle+|T(\nabla\eta)|^2.
\end{eqnarray*}
The last term satisfies the inequality
\begin{eqnarray*}
|T(\nabla\eta)|^2&=&\sum_{i=1}^n\langle T(\nabla\eta),e_i\rangle^2=\sum_{i=1}^n\langle\nabla\eta,T(e_i)\rangle^2\\
&\leq&|\nabla\eta|^2|\sum_{i=1}^n|T(e_i)|^2=|\nabla\eta|^2|T|^2.
\end{eqnarray*}
By setting $T_*=\sup_{\bar{\Omega}}|T|$ and $\eta_0=\sup_{\bar{\Omega}}|\nabla\eta|$ we get
\begin{equation}\label{Tnabla-eta-c3}
|T(\nabla\eta)|^2\leq T_*^2\eta_0^2.
\end{equation}
Furthermore, we set $T_0=\sup_{\bar{\Omega}}|\mathrm{tr}(\nabla T)|$ so that
\begin{eqnarray}\label{Tintegral01-c3}
\nonumber-2\int_{\Omega}u_i^2\langle\mathrm{tr}(\nabla T),T(\nabla\eta)\rangle\dm&\leq&2\int_{\Omega}u_i^2|\mathrm{tr}(\nabla T)||T(\nabla\eta)|\dm\\
\nonumber&\leq&2T_0T_*\eta_0\int_{\Omega}u_i^2\dm\\
&\leq&2T_0T_*\eta_0.
\end{eqnarray}
We also have
\begin{eqnarray*}
\nonumber\|\sum_{i=1}^n\alpha(T(e_i),e_i)\|^2&=&\|\sum_{i=1}^n\sum_{k=n+1}^m\langle\alpha(T(e_i),e_i),e_k\rangle e_k\|^2\\
\nonumber&=&\|\sum_{i=1}^n\sum_{k=n+1}^m\langle A_{e_k}e_i,T(e_i)\rangle e_k\|^2\\
\nonumber&=&\|\sum_{k=n+1}^m\big(\sum_{i=1}^n\langle A_{e_k}e_i,T(e_i)\rangle\big) e_k\|^2\\
&=&\|\sum_{k=n+1}^m\langle A_{e_k},T\rangle e_k\|^2.
\end{eqnarray*}
Estimating the right hand side, we obtain
\begin{eqnarray}\label{Tintegral04-c3}
\nonumber\|\sum_{i=1}^n\alpha(T(e_i),e_i)\|^2&\leq&\sum_{k=n+1}^m|\langle A_{e_k},T\rangle|^2\sum_{k=n+1}^m|e_k|^2\\
\nonumber&\leq&(m-n)\sum_{k=n+1}^m|A_{e_k}|^2|T|^2\\
&\leq&(m-n)^2 A_0^2T_*^2,
\end{eqnarray}
where $A_0=\max\{\sup_{\bar{\Omega}}|A_{e_k}|, k=n+1,\ldots,m\}$, and each $A_{e_k}$ is the Weingarten operator of the immersion with respect to $e_k$.

To complete the proof we need to estimate
\begin{eqnarray}\label{Tintegral02-c3}
\nonumber4\int_{\Omega}u_i\langle\mathrm{tr}(\nabla T),T(\nabla u_i)\rangle\dm&\leq&4\int_{\Omega}|u_i||\mathrm{tr}(\nabla T)||T(\nabla u_i)|\dm\\
\nonumber&\leq&4\left(\int_{\Omega}u_i^2\dm\right)^{\frac{1}{2}}\left(\int_{\Omega}|\mathrm{tr}(\nabla T)|^2|T(\nabla u_i)|^2\dm\right)^{\frac{1}{2}}\\
&\leq&4T_0||T(\nabla u_i)||_{L^2(\Omega,\dm)}
\end{eqnarray}
and
\begin{eqnarray}\label{Tintegral03-c3}
\nonumber-4\int_{\Omega}u_i\langle T(\nabla\eta),T(\nabla u_i)\rangle\dm&\leq& 4\left(\int_{\Omega}|(T(\nabla\eta)|^2|T(\nabla u_i)|^2\dm\right)^{\frac{1}{2}}\\
&\leq&4T_*\eta_0||T(\nabla u_i)||_{L^2(\Omega,\dm)}.
\end{eqnarray}
Substituting \eqref{Tnabla-eta-c3}, \eqref{Tintegral01-c3}, \eqref{Tintegral04-c3}, \eqref{Tintegral02-c3} and \eqref{Tintegral03-c3} into Proposition~\ref{lemmaCYT-c3} we complete the proof of the theorem.
\end{proof}

\section{Drifting Laplacian Case}\label{sec-eta-laplaciano}

In this section, the key to our proofs rely on a slight modification from Theorem~1.2 in \cite{Xia-Xu} namely:

\begin{theorem}[Xia and Xu~\cite{Xia-Xu}]\label{propCYM-c3}
Let $\Omega$ be a domain in an $n$-dimensional complete Riemannian manifold $M$ isometrically immersed in $\mathbb{R}^m$ with mean curvature $H$ and $\lambda_i$ be the $i$-th eigenvalue of the drifting Laplacian. Then
\begin{equation}\label{formCYM-c3}
\sum_{i=1}^k(\upsilon_{k+1}-\upsilon_i)^2 \leq \frac{4}{n} \sum_{i=1}^k (\upsilon_{k+1}-\upsilon_i)\upsilon_i,
\end{equation}
where $\upsilon_i:=\lambda_i +\frac{n^2H_0^2+\eta_0^2+2\bar{\eta_0}}{4}$, $\eta_0=\sup_{\bar{\Omega}}|\nabla \eta|$, $\bar{\eta_0}=\sup_{\bar{\Omega}}|L\eta|$ and $H_0=\sup_{\bar{\Omega}}\|H\|.$
\end{theorem}
\begin{proof}
Taking $T=I$ in Proposition~\ref{lemmaCYT-c3} we obtain
\begin{eqnarray*}
n\sum_{i=1}^k(\lambda_{k+1}-\lambda_i)^2&\leq&\sum_{i=1}^k(\lambda_{k+1}-\lambda_i)\Big[\int_{\Omega}u_i^2\Big(n^2\|H\|^2+|\nabla\eta|^2\Big)\dm\\
&&-4\int_{\Omega}u_i\langle \nabla\eta,\nabla u_i\rangle\dm+4\lambda_i\Big].
\end{eqnarray*}
By setting $\eta_0=\sup_{\bar{\Omega}}|\nabla \eta|$ and $H_0=\sup_{\bar{\Omega}}\|H\|$ we get
\begin{equation}\label{aux-alambda_i+b-c3}
n\sum_{i=1}^k(\lambda_{k+1}-\lambda_i)^2\leq\sum_{i=1}^k(\lambda_{k+1}-\lambda_i)\Big(n^2H_0^2+\eta_0^2+4\lambda_i-4\int_{\Omega}u_i\langle \nabla\eta,\nabla u_i\rangle\dm\Big).
\end{equation}
Since $u_i=0$ on $\partial\Omega$, we can use integration by parts to obtain
\begin{eqnarray}\label{alambda_i+b-c3}
\nonumber-4\int_{\Omega}u_i\g{\nabla \eta}{\nabla u_i}\dm &=&-2\int_{\Omega}\g{\nabla \eta}{\nabla u_i^2}\dm\\
\nonumber&=&2\int_{\Omega}u_i^2L\eta\dm\\
&\leq&2\bar{\eta_0},
\end{eqnarray}
where $\bar{\eta_0}=\sup_{\bar{\Omega}}|L\eta|$. Inequality~\eqref{formCYM-c3} follows by substituting \eqref{alambda_i+b-c3} into \eqref{aux-alambda_i+b-c3}.
\end{proof}

\begin{remark}
The proof of the previous theorem is almost the same as Theorem~1.2 in \cite{Xia-Xu}. We need a different estimate in the final part of the proof, see inequality~\eqref{alambda_i+b-c3}. We point out that inequality~\eqref{formCYM-c3} is the first step to obtain our Theorems \ref{thmA3-c3} and \ref{thmA2-c3} as an application of other known techniques. Moreover, \eqref{formCYM-c3} is an extension for $L$ of the first Yang's inequality, see \cite[p.~6]{Yang} or, alternatively, \cite[Appendix]{Cheng-Yang}. Inequality~\eqref{formCYM-c3} has also been observed by Zeng, see \cite[Equation~(2.2)]{Zeng}. We highlight that he applied the referred inequality for the gradient Ricci soliton case.
\end{remark}

Inequality~\eqref{formCYM-c3} and the fact that $\upsilon_1\leq\upsilon_2\leq\ldots\leq\upsilon_{k+1}$ are positive real numbers corresponds to the assumptions in the next lemma.

\begin{lemma}[Cheng and Yang \cite{Cheng-Yang}]\label{CorCYM-c3}
Let $\eta_1 \leq \eta_2 \leq \ldots \leq \eta_{k+1}$ be positive real numbers satisfying
\begin{equation*}
\sum_{i=1}^k(\eta_{k+1}-\eta_i)^2 \leq \frac{4}{n} \sum_{i=1}^k (\eta_{k+1}-\eta_i)\eta_i,
\end{equation*}
for $n$ positive real number, then
\begin{equation*}
\eta_{k+1} \leq \left( 1+\frac{4}{n}\right)k^{\frac{2}{n}} \eta_1.
\end{equation*}
\end{lemma}

The following corollary is an immediate consequence of Theorem~\ref{propCYM-c3} and Lemma~\ref{CorCYM-c3}, which is analogous to Theorem~2.1 by Zeng~\cite{Zeng}. See also the Euclidean case by Cheng and Yang~\cite{Cheng-Yang} for a very particular estimate.

\begin{corollary}\label{thmA-c3}
Under the assumptions in Theorem~\ref{propCYM-c3}, we have
\begin{equation*}\label{lambda 0-c3}
\upsilon_{k+1}\leq \Big(1+\frac{4}{n}\Big)k^{\frac{2}{n}} \upsilon_1, \quad\mbox{for}\quad k=1,2,\ldots
\end{equation*}
\end{corollary}

From the Weyl's asymptotic formula of eigenvalues \eqref{introd-weyl-c3}, we know that the previous estimate is optimal in the sense of the order on $k$.

In what follows we study some cases from \eqref{formCYM-c3} which were not addressed in \cite{Zeng}. The proof of Theorem~\ref{thmA3-c3} also relies on following result.

\begin{lemma}[Cheng and Yang \cite{Cheng-Yang}] \label{lemA-c3}
Under the assumptions in Lemma~\ref{CorCYM-c3}, we have
\begin{equation*}
F_{k+1}\leq \mathit{C}(n,k)\left(\frac{k+1}{k}\right)^{\frac{4}{n}}F_k,
\end{equation*}
where
\begin{equation*}
F_k:=\left(1+\frac{2}{n}\right)\Lambda_k^2-T_k, \quad\Lambda_k:=\frac{1}{k}\sum_{i=1}^k \eta_i, \quad T_k:=\frac{1}{k}\sum_{i=1}^k \eta_i^2,
\end{equation*}
and
\begin{equation*}
0<\mathit{C}(n,k)=1-\frac{1}{3n}\left(\frac{k}{k+1}\right)^{\frac{4}{n}}\frac{\left(1+\frac{2}{n}\right)\left(1+\frac{4}{n}\right)}{(k+1)^3}<1.
\end{equation*}
\end{lemma}

\subsection{Proof of Theorem~\ref{thmA3-c3}}
\begin{proof}
Using inequality \eqref{formCYM-c3} and Lemma~\ref{lemA-c3}, we obtain
\begin{equation*}
F_{k+1}\leq \mathcal{C}(n,k)\left(\frac{k+1}{k}\right)^{\frac{4}{n}}F_k\leq\left(\frac{k+1}{k}\right)^{\frac{4}{n}}F_k,
\end{equation*}
that is,
\begin{equation*}
\frac{F_{k+1}}{(k+1)^{\frac{4}{n}}}\leq\frac{F_k}{k^{\frac{4}{n}}}.
\end{equation*}
More generally,
\begin{equation}\label{Fk+l-c3}
\frac{F_{k+l}}{(k+l)^{\frac{4}{n}}}\leq\frac{F_k}{k^{\frac{4}{n}}},
\end{equation}
for any positive integer $l$. Furthermore, by Lemma~\ref{lemA-c3} again, we have
\begin{eqnarray}\label{2.70-1-c3}
\nonumber\frac{F_k}{k^{\frac{4}{n}}}&=&\frac{\left(1+\frac{2}{n}\right)\Lambda_k^2-T_k}{k^{\frac{4}{n}}}=\frac{\frac{2}{n}\Big(\frac{1}{k}\sum_{i=1}^k \upsilon_i\Big)^2-\frac{1}{k}\sum_{i=1}^k(\upsilon_i-\frac{1}{k}\sum_{i=1}^k \upsilon_i)}{k^{\frac{4}{n}}}\\
&\leq&\frac{\frac{2}{n}\Big(\frac{1}{k}\sum_{i=1}^k \upsilon_i\Big)^2}{k^{\frac{4}{n}}}
\end{eqnarray}
and
\begin{equation}\label{limetaquadrado1}
\frac{F_{k+l}}{(k+l)^{\frac{4}{n}}}=\Big(1+\frac{2}{n}\Big)\left(\frac{\frac{1}{k+l}\sum_{i=1}^{k+l}\upsilon_i}{(k+l)^{\frac{2}{n}}}\right)^2
-\frac{\frac{1}{k+l}\sum_{i=1}^{k+l}\upsilon_i^2}{(k+l)^{\frac{4}{n}}}.
\end{equation}
From Weyl's asymptotic formula, one has
\begin{equation}\label{limlambda}
\lim_{k\rightarrow\infty}\frac{\frac{1}{k}\sum_{i=1}^k\lambda_i}{k^{\frac{2}{n}}}=\frac{n}{n+2}\frac{4\pi^2}{(\omega_n \mathrm{vol}\,\Omega)^\frac{2}{n}}
\end{equation}
and
\begin{equation}\label{limlambdaquadrado}
\lim_{k\rightarrow\infty}\frac{\frac{1}{k}\sum_{i=1}^k\lambda_i^2}{k^{\frac{4}{n}}}=\frac{n}{n+4}\frac{16\pi^4}{(\omega_n \mathrm{vol}\,\Omega)^\frac{4}{n}},
\end{equation}
then \eqref{limlambda} implies
\begin{align*}
\lim_{k\rightarrow\infty}\frac{\frac{1}{k}\sum_{i=1}^k\lambda_i}{k^{\frac{2}{n}}}&=\lim_{k\rightarrow\infty}\frac{\Big(\frac{1}{k}\sum_{i=1}^k\lambda_i\Big)
+\frac{n^2H_0^2+\eta_0^2+2\bar{\eta_0}}{4}}{k^{\frac{2}{n}}}\\
&=\lim_{k\rightarrow\infty}\frac{\frac{1}{k}\sum_{i=1}^k(\lambda_i+\frac{n^2H_0^2+\eta_0^2+2\bar{\eta_0}}{4})}{k^{\frac{2}{n}}}.
\end{align*}
Therefore
\begin{equation}\label{limeta}
\lim_{k\rightarrow\infty}\frac{\frac{1}{k}\sum_{i=1}^k\upsilon_i}{k^{\frac{2}{n}}}=\frac{n}{n+2}\frac{4\pi^2}{(\omega_n
\mathrm{vol}\,\Omega)^\frac{2}{n}}.
\end{equation}
Analogously, from \eqref{limlambdaquadrado} we get
\begin{align*}
&\lim_{k\rightarrow\infty}\frac{\frac{1}{k}\sum_{i=1}^k\Big(\lambda_i+\frac{n^2H_0^2+\phi_0^2+4\phi_0}{4(1+ \phi_0)}\Big)^2}{k^{\frac{4}{n}}}\\
=&\frac{n}{n+4}\frac{16\pi^4}{(\omega_n vol\Omega)^\frac{4}{n}}+\lim_{k\rightarrow\infty}\frac{\big(\frac{n^2H_0^2+\phi_0^2+4\phi_0}{4(1+ \phi_0)}\big)^2}{k^{\frac{2}{n}}}
+\lim_{k\rightarrow\infty}\frac{2\frac{n^2H_0^2+\phi_0^2+4\phi_0}{4(1+ \phi_0)}}{k^{\frac{2}{n}}}\frac{\frac{1}{k}\sum_{i=1}^k\lambda_i}{k^{\frac{2}{n}}}.
\end{align*}
Hence
\begin{equation}\label{limetaquadrado}
\lim_{k\rightarrow\infty}\frac{\frac{1}{k}\sum_{i=1}^k\upsilon_i^2}{k^{\frac{4}{n}}}=\frac{n}{n+4}\frac{16\pi^4}{(\omega_n \mathrm{vol}\,\Omega)^\frac{4}{n}}.
\end{equation}
Thus, from \eqref{limetaquadrado1}, \eqref{limeta} and \eqref{limetaquadrado} we obtain
\begin{align}\label{limetaquadrado1o}
\nonumber\lim_{l\rightarrow\infty}\frac{F_{k+l}}{(k+l)^{\frac{4}{n}}}&=\lim_{l\rightarrow\infty}\Big(1+\frac{2}{n}\Big)\left(\frac{\frac{1}{k+l}\sum_{i=1}^{k+l}\upsilon_i}
{(k+l)^{\frac{2}{n}}}\right)^2-\lim_{l\rightarrow\infty}\frac{\frac{1}{k+l}\sum_{i=1}^{k+l}\upsilon_i^2}{(k+l)^{\frac{4}{n}}}\\
&=\frac{2n}{(n+2)(n+4)}\frac{16\pi^4}{(\omega_n \mathrm{vol}\,\Omega)^\frac{4}{n}}.
\end{align}
From \eqref{Fk+l-c3}, \eqref{2.70-1-c3} and \eqref{limetaquadrado1o} we conclude that
\begin{equation*}
\frac{\frac{2}{n}\Big(\frac{1}{k}\sum_{i=1}^k \upsilon_i\Big)^2}{k^{\frac{4}{n}}}\geq\frac{2n}{(n+2)(n+4)}\frac{16\pi^4}{(\omega_n \mathrm{vol}\,\Omega)^\frac{4}{n}}
\end{equation*}
for any positive integer $k$. Note that this is sufficient to complete our proof.
\end{proof}

\subsection{Proof of Theorem~\ref{thmA2-c3}}
\begin{proof}
From \eqref{formCYM-c3} we write
\begin{equation*}
\sum_{i=1}^k(\upsilon_{k+1}-\upsilon_i)(\upsilon_{k+1}-\upsilon_i)-\frac{4}{n} \sum_{i=1}^k(\upsilon_{k+1}-\upsilon_i)\upsilon_i\leq 0,
\end{equation*}
thus
\begin{equation*}
\sum_{i=1}^k(\upsilon_{k+1}-\upsilon_i)\upsilon_{k+1}\leq\Big(1+\frac{4}{n}\Big) \sum_{i=1}^k (\upsilon_{k+1}-\upsilon_i)\upsilon_i.
\end{equation*}
We affirm that
\begin{equation}\label{etak+1<etai-c3}
\upsilon_{k+1}\leq\Big(1+\frac{4}{n}\Big)\upsilon_i,
\end{equation}
for all $i=1,\ldots,k$. Indeed, otherwise there would be $i_0$ such that
\begin{equation*}
\upsilon_{k+1}>\Big(1+\frac{4}{n}\Big)\upsilon_{i_0}>\Big(1+\frac{4}{n}\Big)\upsilon_{i_0-1}>\ldots>\Big(1+\frac{4}{n}\Big)\upsilon_1
\end{equation*}
which would imply
\begin{equation*}
\sum_{i=1}^{i_0}(\upsilon_{i_0+1}-\upsilon_i)\upsilon_{i_0+1}>\Big(1+\frac{4}{n}\Big) \sum_{i=1}^{i_0} (\upsilon_{i_0}-\upsilon_i)\upsilon_i
\end{equation*}
and would contradict inequality \eqref{formCYM-c3}. Summing up the terms on the right hand side of \eqref{etak+1<etai-c3} we deduce
\begin{equation*}
\upsilon_{k+1}\leq\frac{1}{k}\Big(1+\frac{4}{n}\Big)\sum_{i=1}^k\upsilon_i
\end{equation*}
which proves \eqref{cor.ineq1-c3}. To prove \eqref{cor.ineq2-c3}, we note that \eqref{formCYM-c3} is equivalent to
\begin{equation*}
\mathcal{P}(\upsilon_{k+1})=k(\upsilon_{k+1})^2-\upsilon_{k+1}\Big(2+\frac{4}{n}\Big)\sum_{i=1}^k\upsilon_i+\Big(1+\frac{4}{n}\Big)\sum_{i=1}^k(\upsilon_i)^2\leq 0
\end{equation*}
so that the discriminant of $\mathcal{P}(\upsilon_{k+1})$ satisfies
\begin{equation}\label{discriminante-c3}
\mathcal{D}=\Big(2+\frac{4}{n}\Big)^2\Big(\sum_{i=1}^k\upsilon_i\Big)^2-4k\Big(1+\frac{4}{n}\Big)\sum_{i=1}^k(\upsilon_i)^2\geq 0.
\end{equation}
Since $\mathcal{P}(\upsilon_{k+1})\leq 0$ we have $r_{k+1}^{\eta}\leq\upsilon_{k+1}\leq R_{k+1}^{\eta}$, where $r_{k+1}^{\eta}$ and $R_{k+1}^{\eta}$ are the smaller and the biggest root of $\mathcal{P}$,  respectively. Then
\begin{equation}\label{ineq Baskara-c3}
\upsilon_{k+1}\leq R_{k+1}^{\eta}=\frac{1}{2k}\Big[\Big(2+\frac{4}{n}\Big)\sum_{i=1}^k\upsilon_i+\sqrt{\mathcal{D}}\Big].
\end{equation}
Substituting \eqref{discriminante-c3} into \eqref{ineq Baskara-c3} we obtain
\begin{eqnarray*}
\upsilon_{k+1}&\leq&\frac{1}{k}\Big(1+\frac{2}{n}\Big)\sum_{i=1}^k\upsilon_i+\Big[\Big(\frac{1}{k}+\frac{2}{kn}\Big)^2\Big(\sum_{i=1}^k\upsilon_i\Big)^2
-\frac{1}{k}\Big(1+\frac{4}{n}\Big)\sum_{i=1}^k(\upsilon_i)^2\Big]^\frac{1}{2}\\
&=&\frac{1}{k}\Big(1+\frac{2}{n}\Big)\sum_{i=1}^k\upsilon_i+\Big[\Big(\frac{2}{kn}\sum_{i=1}^k\upsilon_i\Big)^2+\frac{1}{k^2}
\Big(1+\frac{4}{n}\Big)\Big(\sum_{i=1}^k\upsilon_i\Big)^2\\
&&-\frac{1}{k}\Big(1+\frac{4}{n}\Big)\sum_{i=1}^k(\upsilon_i)^2\Big]^\frac{1}{2},
\end{eqnarray*}
or equivalently
\begin{align*}
\upsilon_{k+1}\leq&\frac{1}{k}\Big(1+\frac{2}{n}\Big)\sum_{i=1}^k\upsilon_i+\Big[\Big(\frac{2}{kn}\sum_{i=1}^k\upsilon_i\Big)^2
-\frac{1}{k}\Big(1+\frac{4}{n}\Big)\Big(\sum_{i=1}^k(\upsilon_i)^2-\frac{1}{k}\Big(\sum_{i=1}^k\upsilon_i\Big)^2\Big)\Big]^\frac{1}{2}\\
=&\frac{1}{k}\Big(1+\frac{2}{n}\Big)\sum_{i=1}^k\upsilon_i+\Big[\Big(\frac{2}{kn}\sum_{i=1}^k\upsilon_i\Big)^2\\
&-\frac{1}{k}\Big(1+\frac{4}{n}\Big)\Big(\sum_{i=1}^k(\upsilon_i)^2-\frac{2}{k}\Big(\sum_{i=1}^k\upsilon_i\Big)^2+\frac{1}{k}\Big(\sum_{i=1}^k\upsilon_i\Big)^2\Big)\Big]^\frac{1}{2}\\
=&\frac{1}{k}\Big(1+\frac{2}{n}\Big)\sum_{i=1}^k\upsilon_i+\Big[\Big(\frac{2}{kn}\sum_{i=1}^k\upsilon_i\Big)^2\\
&-\frac{1}{k}\Big(1+\frac{4}{n}\Big)\Big(\sum_{i=1}^k(\upsilon_i)^2-\frac{2}{k}\sum_{i,j=1}^k\upsilon_i\upsilon_j+\frac{1}{k}\Big(\sum_{i=1}^k\upsilon_i\Big)^2\Big)\Big]^\frac{1}{2},
\end{align*}
so
\begin{equation*}
\upsilon_{k+1}\leq\frac{1}{k}\Big(1+\frac{2}{n}\Big)\sum_{i=1}^k\upsilon_i+\Big[\Big(\frac{2}{kn}\sum_{i=1}^k\upsilon_i\Big)^2
-\frac{1}{k}\Big(1+\frac{4}{n}\Big)\sum_{j=1}^k\Big(\upsilon_j-\frac{1}{k}\sum_{i=1}^k\upsilon_i\Big)^2\Big]^\frac{1}{2}
\end{equation*}
which proves \eqref{cor.ineq2-c3}. Finally, as \eqref{formCYM-c3} is true for all $k$, it follows that
\begin{equation*}
\sum_{i=1}^k(\upsilon_k-\upsilon_i)^2\leq\frac{4}{n} \sum_{i=1}^k
(\upsilon_k-\upsilon_i)\upsilon_i,
\end{equation*}
i.e., we can observe again that the polynomial $\mathcal{P}(\upsilon_k)\leq 0$. Analogously, we have
\begin{equation}\label{cor-eta_k-c3}
\upsilon_k\geq r_{k}^{\eta}=\frac{1}{k}\Big(1+\frac{2}{n}\Big)\sum_{i=1}^k\upsilon_i-\Big[\Big(\frac{2}{kn}\sum_{i=1}^k\upsilon_i\Big)^2
-\frac{1}{k}\Big(1+\frac{4}{n}\Big)\sum_{j=1}^k\Big(\upsilon_j-\frac{1}{k}\sum_{i=1}^k\upsilon_i\Big)^2\Big]^\frac{1}{2}.
\end{equation}
Inequality~\eqref{cor.ineq3-c3} follows from \eqref{cor.ineq2-c3} and \eqref{cor-eta_k-c3}.
\end{proof}

\vspace{0.2cm}
\noindent\textbf{Acknowledgements:}
The authors would like to express their sincere thanks to referee for his careful reading and useful comments which improved the paper. They also would like to thank the Professor Dragomir Mitkov Tsonev for useful comments, discussions and constant encouragement. The first author would like to thank the Department of Mathematics at Lehigh University, where part of this work was carried out. He is grateful to Huai-Dong Cao and Mary Ann for their warm hospitality and constant encouragement. The first author is partially supported by grant 202234/2017-7, Conselho Nacional de Desenvolvimento Cient\'ifico e Tecnol\'ogico (CNPq), of the Ministry of Science, Technology and Innovation of Brazil.

\end{document}